\documentclass{amsart}
\usepackage{latexsym,amssymb}
\usepackage[english]{babel}

\usepackage{amsfonts, amsthm}

\def\Z{\mathbb{Z}}
\def\N{\mathbb{N}}
\def\r{\mathsf{r}}
\def\p{\mathsf{p}}
\def\ap{\mathsf{ap}}

\def\BHR{\mathop{\rm BHR}}
\def\MPP{\mathop{\rm MPP}}
\def\nf{near $1$-factor }

\newtheorem{defi}{Definition}
\newtheorem{prop}[defi]{Proposition}
\newtheorem{lem}[defi]{Lemma}
\newtheorem{rem}[defi]{Remark}

\newtheorem{corollar}[defi]{Corollary}
\newtheorem{thm}[defi]{Theorem}
\newtheorem*{conj}{Conjecture}
\theoremstyle{definition}
\newtheorem{ex}[defi]{Example}

\begin{document}

\title[A generalization of the problem of Mariusz Meszka]{A generalization of the
problem\\ of Mariusz Meszka}

\author{Anita Pasotti}
\address{DICATAM - Sez. Matematica, Universit\`a degli Studi di
Brescia, Via
Branze 43, I-25123 Brescia, Italy}
\email{anita.pasotti@unibs.it}

\author{Marco Antonio Pellegrini}
\address{Dipartimento di Matematica e Fisica, Universit\`a Cattolica del Sacro Cuore, Via
Musei 41,
I-25121 Brescia, Italy}
\email{marcoantonio.pellegrini@unicatt.it}

\begin{abstract}
Mariusz Meszka has conjectured that given a prime $p=2n+1$ and a list $L$ containing $n$
positive integers not exceeding $n$  there exists a near $1$-factor in $K_p$ whose list of edge-lengths is $L$.
In this paper we propose a generalization of this problem to the case in which $p$ is an odd integer not necessarily prime.
In particular, we give a necessary condition for the existence of such a near $1$-factor for any odd integer $p$. We show that this
condition is also sufficient for any list $L$ whose underlying set $S$ has size $1$, $2$, or $n$.
Then we prove that the conjecture is true if
 $S=\{1,2,t\}$ for any positive integer $t$ not coprime with the order $p$ of the complete graph.
 Also, we give partial results when $t$ and $p$ are coprime. Finally, we present a complete solution for $t\leq 11$.
\end{abstract}

 \maketitle

\section{Introduction}

Throughout this paper $K_v$ will denote the complete graph on $\{0,1,\dots,v-1\}$ for any positive integer $v$.
For the basic terminology on graphs we refer to \cite{Wbook}.
Following \cite{HR}, we define the {\it length} $\ell(x,y)$ of an edge $[x,y]$ of $K_v$ as
$$\ell(x,y)=min(|x-y|,v-|x-y|).$$
If $\Gamma$ is any subgraph of $K_v$, then the list of edge-lengths of $\Gamma$ is the multiset
$\ell(\Gamma)$ of the lengths (taken with their respective multiplicities) of  all the edges
of $\Gamma$.
The set of the edges of $\Gamma$ will be denoted by $E(\Gamma)$.
Also, by $\delta(\Gamma)$ we will mean the multiset
$$\delta(\Gamma)=\{|x-y| :  [x,y]\in E(\Gamma)\}.$$
Clearly, if all the elements of $\delta(\Gamma)$ do not exceed $\frac{v-1}{2}$
it results $\delta(\Gamma)=\ell(\Gamma)$.
For our convenience, if a list $L$ consists of
$a_1$ $1'$s, $a_2$ $2'$s, \ldots, $a_t$ $t'$s,
we will write $L=\{1^{a_1},2^{a_2},\ldots,t^{a_t}\}$, whose \emph{underlying set} is
 the set
$\{1,2,\ldots,t\}$.

The following conjecture \cite{BM,W} is due to Marco Buratti
(2007, communication to Alex Rosa).

\begin{conj}[Buratti]
For any prime $p=2n+1$
and any list $L$ of $2n$ positive integers not exceeding $n$, there exists a
Hamiltonian path $H$ of $K_p$ with $\ell(H)=L$.
\end{conj}

In \cite{HR} Peter Horak and Alex Rosa generalized Buratti's conjecture.
Such a generalization has been restated in a easier form in \cite{1235} as follows.

\begin{conj}[Horak and Rosa]\label{HR iff B}
Let $L$ be a list of $v-1$ positive integers not exceeding $\lfloor{v\over 2}\rfloor$.
Then there exists a Hamiltonian path $H$ of $K_v$ such that $\ell(H)=L$ if, and only if, the following condition holds:
\begin{equation}\label{B}
\left. \begin{array}{c}
\textrm{for any divisor $d$ of $v$, the number of multiples of $d$} \\
\textrm{appearing in $L$ does not exceed $v-d$.}
\end{array}\right.
\end{equation}
\end{conj}

Following \cite{1235}, by $\BHR(L)$ we will denote the above conjecture for a given list
$L$.
Some partial results have been obtained about this problem, see \cite{CDF,DJ,HR,1235,12t},
but the conjecture is still wide open.\\
Recently, Mariusz Meszka formulated a very similar conjecture
concerning near $1$-factors. We recall that a \emph{near $1$-factor} of $K_{2n+1}$ is a subgraph of $K_{2n+1}$ consisting
in $n$ disjoint edges and one isolated vertex, while a \emph{$1$-factor} of $K_{2n}$ is a subgraph of $K_{2n}$ consisting
in $n$ disjoint edges.

\begin{conj}[Meszka]
For any prime $p=2n+1$
and any list $L$ of $n$ positive integers not exceeding $n$, there exists a
near $1$-factor $F$ of $K_p$ with $\ell(F)=L$.
\end{conj}

To best of our knowledge,
all the results on the problem proposed by Meszka are contained in \cite{R}, where
 Rosa proved that the conjecture is true when the elements of the list are
$\{1,2,3\}$ or $\{1,2,3,4\}$ or $\{1,2,3,4,5\}$. Also he proved that the conjecture is true
when the list $L$ contains a sufficient number of $1$'s, in detail, if $L=\{1^{a_1},2^{a_2},\ldots,n^{a_n}\}$
with $a_1\geq \frac{n^2}{2}$.
Moreover, with the aid of a computer, we have verified the validity of Meszka's conjecture
for all primes $p\leq 23$.\medskip

Working on this conjecture it is easy to see that the assumption $p$ prime is not necessary,
as it happens for Buratti's one. Indeed we found another condition which is necessary
and which led us to propose our conjecture.

\begin{conj}
Let $v=2n+1$ be an odd integer and  $L$ be a list of $n$ positive integers not
exceeding $n$.
Then there exists a near $1$-factor $F$ of $K_v$ such that $\ell(F)=L$ if, and only if, the following condition holds:
\begin{equation}\label{PP}
\left. \begin{array}{c}
\textrm{for any divisor $d$ of $v$, the number of multiples of $d$} \\
\textrm{appearing in $L$ does not exceed $\frac{v-d}{2}$.}
\end{array}\right.
\end{equation}
\end{conj}

With the acronym $\MPP$, which stands for Meszka-Pasotti-Pellegrini, we will denote this more general conjecture.
 In particular $\MPP(L)$  will denote the conjecture for a given list $L$.
 Clearly, if $v$ is a prime then our conjecture reduces to Meszka's one. With the aid of a computer, we have verified its validity
for all odd integers $v\leq 23$. We point out that in the statement, the
actual conjecture is the sufficiency.
In fact we can prove that condition (\ref{PP}) is necessary.

\begin{prop}\label{nec}
The list $L$ of edge-lengths of any near $1$-factor of $K_v$ satisfies condition
\emph{(\ref{PP})}.
\end{prop}

\begin{proof}
Let $F$ be a near $1$-factor of $K_v$ with $\ell(F)=L$ and let $d$ be a divisor of $v$.
Denote by $D$ the sublist of $L$ containing all the multiples of $d$ appearing in $L$, hence we have to prove that
$|D|\leq \frac{v-d}{2}$. Note that if the length of an edge is a multiple of $d$, the vertices
of such an edge have to be in the same residue class modulo $d$ and that
$F$ has exactly  $\frac{v}{d}$
vertices in each residue class modulo $d$. Also, $\frac{v}{d}$
is an odd number and this implies that with the elements
of each residue class modulo $d$ we can construct at most $\frac{\frac{v}{d}-1}{2}$ edges whose length is in $D$.
Since there are exactly $d$ residue classes modulo $d$ we have at most $\frac{v-d}{2}$
edges whose length is in $D$, namely $|D|\leq \frac{v-d}{2}$.
\end{proof}

Before giving the main result of this paper we would like to show some connections between
$\MPP$-problem
and \emph{graph decompositions}, as  done for $\BHR$-problem in \cite{1235}. For a general
background
on graph decompositions see \cite{BE}.
Reasoning in the same way as we have done in \cite{1235} for the $\BHR$-problem, one can easily
obtain that
$\MPP(L)$ can be reformulated in the following way.

\begin{conj}
A Cayley multigraph {\it $Cay[\Z_v:\Lambda]$ admits a cyclic decomposition into
near $1$-factors if and only if $\Lambda=L \ \cup \ -L$
with $L$ satisfying condition $(\ref{PP})$}.
\end{conj}

For reader convenience we recall the definition of a Cayley multigraph.
A list $\Omega$ of elements of an additive group $G$ is
said to be \emph{symmetric} if $0 \notin \Omega$ and the multiplicities
of $g$ and $-g$ in $\Omega$ coincide for any $g \in G$.
If $\Omega$ does not have repeated elements then one can consider the \emph{Cayley
graph on $G$ with connection set $\Omega$}, denoted $Cay[G:\Omega]$, whose
vertex-set is $G$  and in which $[x,y]$ is an edge if and only if $x-y \in \Omega$.
Cayley graphs have a great importance in combinatorics
and they are precisely the graphs admitting an automorphism
group acting sharply transitively on the vertex-set (see, e.g., \cite{GR}).
If, more generally, the symmetric list $\Omega$ has repeated
elements one can consider the \emph{Cayley multigraph on $G$ with
connection multiset $\Omega$}, also denoted $Cay[G:\Omega]$
and with vertex-set $G$, where the multiplicity of an edge $[x,y]$ is
the  multiplicity of $x-y$ in $\Omega$ (see, e.g., \cite{BMnew}).

In the next sections of this paper we prove some results concerning the $\MPP$-problem that we can
summarize in the following theorem.

\begin{thm}\label{th2}
Let $L$ be a list of $n$ positive integers not exceeding $n$ and let $v=2n+1$.
Then, $\MPP(L)$ holds whenever the underlying set $S$ of $L$ satisfies one of the following conditions:
\begin{itemize}
\item[\rm{(1)}] $|S|=1,2$ or $n$;
\item[\rm{(2)}] $S=\{1,2,t\}$, where $t\geq 3$ is not coprime with $v$.
\end{itemize}
Furthermore, $\MPP(L)$ holds if $L=\{1^a, 2^b, t^ c\}$ with $a+b\geq \lfloor\frac{t-1}{2}\rfloor$.
\end{thm}

The proof of Theorem \ref{th2} will follow from Propositions \ref{1n}, \ref{2lunghezze}, \ref{big} and \ref{12t-notcop}.

If $S=\{1,2,t\}$ and $\gcd(v,t)=1$, we also present some  partial results for the cases not covered by Theorem \ref{th2}.
In particular, we give a complete solution for $t\leq 11$.

\section{Linear realizations and their relationship with Skolem sequences}

A \emph{cyclic realization} of a list $L$ of $n$ positive integers not exceeding $n$
is a near $1$-factor $F$ of $K_{2n+1}$ such that
$\ell(F)=L$.
For example, given $L=\{1^2,4^3,6\}$ the near $1$-factor $F=\{[0,4],[1,2],[5,12],[6,10],[7,11],[8,9]\}\cup\{3\}$ of $K_{13}$
is a cyclic realization of $L$. Clearly if a list $L$ admits a cyclic realization we can say that
$\MPP(L)$ is true.\\
A \emph{linear realization} of a list $L$ with $n$ positive integers not exceeding $2n$
is a near $1$-factor $F$ of $K_{2n+1}$ such that
$\delta(F)=L$.
It is quite immediate that if the elements of $L$ do not exceed $n$ a linear realization of $L$ is nothing but a cyclic realization,
namely a near $1$-factor $F$ of $K_{2n+1}$ with $\delta(F)=\ell(F)=L$.
Following \cite{R}, we will say that a linear realization of a list $L$ is \emph{almost perfect} or \emph{perfect} if the isolated vertex
of the near $1$-factor is $2n-1$ or $2n$, respectively.\\
It is important to underline that there is a strong relationship between (almost) perfect
 linear realizations and Skolem sequences. We point out that Skolem sequences and their generalizations 
 (see, e.g., \cite{BA}) have revealed to be very useful in the construction of several kinds of combinatorial designs
(see, for example, the survey \cite{FM} and the references therein as well as \cite{BBRT,BU,WB}).
 In order to present this connection we recall the basic definitions, see \cite{S}.
 A \emph{Skolem sequence} of order $n$ is a sequence $S=(s_1,s_2,\ldots,s_{2n})$
of $2n$ integers satisfying the following conditions:
\begin{itemize}
\item[(1)] for every $k\in\{1,2,\ldots,n\}$ there exist exactly two elements $s_i,s_j\in S$ such that $s_i=s_j=k$;
\item[(2)] if $s_i=s_j=k$ with $i<j$, then $j-i=k$.
\end{itemize}
It is worth to observe that a Skolem sequence can also be written as a collection of ordered pairs
$\{(a_i,b_i)\ :\ 1\leq i\leq n,\ b_i-a_i=i \}$ with $\bigcup_{i=1}^{n}\{a_i,b_i\}=\{0,1,\ldots,2n-1\}$.
For instance, the Skolem sequence $S=(1,1,3,4,5,3,2,4,2,5)$ of order $5$ can be seen as the set 
$\{(0,1),(6,8),(2,5),(3,7),(4,9)\}$.
Hence we can conclude that a Skolem sequence is equivalent to  a perfect linear realization of a \underline{set}. 
For example, $S'=(1,1,3,4,5,3,2,4,2,5,0)$, obtained from $S$ adding $0$ at the end, is a perfect linear realization of the set $\{1,2,3,4,5\}$.
Also, a \emph{hooked Skolem sequence} of order $n$ is a sequence $HS=(s_0,s_1,\ldots,s_{2n})$
of $2n+1$ integers satisfying above conditions (1) and (2) and such that $s_{2n-1}=0$.
So we have that a hooked Skolem sequence is an almost perfect linear realization of a \underline{set}.\\
Hence, we can say that $\MPP$-problem  can be view as a generalization of Skolem sequences and
we propose the following more general definition.
Let $L=\{1^{a_1},2^{a_2},\ldots,$ $n^{a_n}\}$ be a list on the set $\{1,2,\ldots,n\}$
and let $k$ be an element of the set $\{1,2,\ldots,2n+1\}$.
We call $k$-\emph{extended Skolem sequence} of $L$ any sequence $S=(s_0,s_1,\ldots,s_{2n})$
for which it is possible to partition $\{1,2,\ldots,2n+1\}\setminus \{k\}$ into a set $T$
of $n$ ordered pairs $(x,y)$ with $x<y$ such that the set $T_i:=\{(x,y)\in T\ |\ s_x=s_y=y-x=i\}$
has size $a_i$ for $1\leq i\leq n$.  It is clear that a linear realization of $L$
is a $k$-extended Skolem sequence of $L$ for a suitable $k$. Also, it is
perfect or almost perfect when $k=2n$ or $2n-1$, respectively.
In these cases, in view of the classical definitions given above of Skolem sequences, one may speak
of an ordinary or hooked Skolem sequence of $L$, respectively.
For example, given $L_1=\{1,3,6^2\}$ the near $1$-factor $F_1=\{[0,6],[1,7],[3,4],[5,8]\}\cup\{2\}$ of $K_9$
is a linear realization of $L_1$. The corresponding $2$-extended Skolem sequence is
$S_1=(6_1,6_2,0,1,1,3,6_1,6_2,3)$, where we use $6_1$ and $6_2$ to distinguish the same length $6$ belonging to distinct pairs.
Take now
the near $1$-factor $F_2=\{[0,3],[1,4],[2,6]\}\cup\{5\}$ of $K_7$, it is an almost perfect realization of $L_2=\{3^ 2,4\}$
and
it is easy to see that the corresponding hooked Skolem sequence is
$S_2=(3_1,3_2,4,3_1,3_2,0,4)$. \\
Clearly a list $L$ can admit both a  not perfect linear realization and a perfect linear realization.
For instance $F_3=\{[0,6],[1,7],[2,5],[3,4]\}\cup\{8\}$ is a \emph{perfect} linear realization of the above list
$L_1=\{1,3,6^ 2\}$,
which corresponds to the Skolem sequence $S_3=(6_1,6_2,3,1,1,3,6_1,6_2,0)$.
It is easy to see that if there exists a perfect linear realization of a list $L$,
then there exists a $1$-factor $F$ of $K_{2|L|}$ such that $\delta(F)=L$.

 For convenience, in the following, by $\r L$, $\ap L$ and $\p L$
we will denote a linear realization, an almost perfect linear realization and a perfect linear realization of $L$,
respectively.

Given two lists $L_1$ and $L_2$ it is possible to obtain a linear realization of the list
$L_1\cup L_2$ \emph{composing} the linear realizations of $L_1$ and $L_2$.
We have to point out that it is not always possible to compose two linear realizations, since the existence
of the composition depends on the properties of the two realizations as shown in the following lemma, see \cite{R}.
When the composition exists it will be denoted by ``$+$''.

\begin{lem}\label{LemmaComposition}
Given two lists $L_1$ and $L_2$ we have:
\begin{itemize}
\item[\rm{(1)}] $\p L_1+\p L_2=\p(L_1\cup L_2)$;
\item[\rm{(2)}] $\p L_1+\ap L_2=\ap(L_1\cup L_2)$;
\item[\rm{(3)}] $\ap L_1+\ap L_2=\p(L_1\cup L_2)$;
\item[\rm{(4)}] $\p L_1+\r L_2=\r (L_1\cup L_2)$.
\end{itemize}
\end{lem}

\begin{proof}
(1) (2) (3) These items have been proved by Rosa, see \cite{R}.\\\
(4) Let $|L_1|=w$ and $|L_2|=z$. Let $S_1=(\ell_1,\ell_2,\ldots,\ell_{2w},0)$ be the sequence corresponding to
a perfect linear realization of $L_1$ and let
$S_2=(\bar\ell_1,\bar\ell_2,\ldots,\bar\ell_i,0,$ $\bar\ell_{i+1}, \ldots,\bar\ell_{2z})$ be the sequence corresponding to
a linear realization of $L_2$. Then $S=(\ell_1,\ell_2,\ldots,\ell_{2w},$ $\bar\ell_1,\bar\ell_2,\ldots,\bar\ell_i,0,\bar\ell_{i+1},\ldots,\bar\ell_{2z})$
is the sequence corresponding to a linear realization of $L_1\cup L_2$.
\end{proof}

\begin{ex}
Consider the previous linear realizations of the lists $L_1=\{1,3,$ $6^ 2\}$ and $L_2=\{3^ 2,4\}$. Using Lemma \ref{LemmaComposition} one gets a perfect
realization of $\{1^ 2,3^ 2,$ $6^ 4 \}$, an almost perfect realization of $\{1,3^ 3,4,6^ 2\}$, a perfect realization of $\{3^ 4,4^ 2\}$
and a linear realization of $\{1^ 2,3^ 2,6^ 4 \}$. Indeed, $\p \{1^ 2,3^ 2,6^ 4 \}=\p L_1+\p L_1$ can be obtained
using twice the sequence $S_3$ to get the sequence $(6_1,6_2,3_1,$ $1_1,1_1,3_1,6_1,$ $6_2,6_3,6_4,3_2,1_2,1_2,3_2,6_3,6_4,0)$.
Using the sequences $S_3$ and $S_2$ one gets the sequence $(6_1,6_2,3_1,1,1,3_1,6_1,6_2,3_2,3_3,4,3_2,3_3,0,4)$, corresponding to
$\ap\{1,3^ 3, 4,$ $6^ 2\}=\p L_1+\ap L_2$. Now, using twice the sequence $S_2$, one gets the sequence $(3_1,3_2,4_1,3_1,3_2,4_2,4_1,3_3,3_4,4_2,3_3,3_4,0)$
corresponding to $\p \{3^ 4,4^ 2\}= \ap L_2+\ap L_2$. Finally, using the sequences $S_3$ and $S_1$, one gets the sequence
$(6_1,6_2,3_1,1_1,1_1,3_1,6_1,$ $6_2,6_3,6_4,0,1_2,1_2,3_2,6_3,6_4,3_2)$, corresponding to $\r \{1^ 2, 3^ 2,6^ 4\}= \p L_1+\r L_1$.
\end{ex}

Given a list $L$ and a positive integer $q$, by $q\cdot \p L$ we will mean the perfect linear realization
$\underbrace{\p L+\p L+\ldots +\p L}_{q \ \mathrm{  times}}$. \\
In view of the above lemma we will look for linear realizations possibly (almost) perfect.

\begin{corollar}\label{odd}
There exists a perfect linear realization of $L=\{(2i+1)^{a_i}\ |\  i=0,\ldots,n, \forall n \in \mathbb{N}, a_i\geq a_j\ if\ i<j\}$.
\end{corollar}

\begin{proof}
Given an odd integer $2k+1$, it is immediate that $S_k=(2k+1,2k-1,2k-3,\ldots,9,7,5,3,1,1,3,5,7,9,\ldots,2k-3,2k-1,2k+1,0)$
corresponds to a perfect linear realization of the list $L_k=\{1,3,5,\ldots,2k-1,2k+1\}$.
Hence, by Lemma \ref{LemmaComposition} a linear realization of $L$ is $\p L=a_n \cdot \p L_n + \sum_{i=0}^{n-1} (a_i-a_{i+1})\cdot \p L_i$.
\end{proof}

\begin{lem}[A. Rosa, \cite{R}]\label{l^l}
The list $\{x^x\}$ admits a perfect linear realization for each $x \geq 1$.
\end{lem}

\begin{lem}
The list $L=\{x^{x-1}\}$ admits a linear realization for each $x \geq 1$.
If $x=2$, such a realization is almost perfect.
\end{lem}

\begin{proof}
The near $1$-factor $F=\{[i,i+x]\ |\ 0\leq i\leq x-2\}\cup\{x-1\}$
of $K_{2x-1}$ is such that
$\delta(F)=L$.
If $x=2$ the statement immediately follows.
\end{proof}

\begin{corollar}
If there exists a linear realization of a list $L=\{\ell_1^{a_1}, \ell_2^{a_2},\ldots,\ell_n^{a_n}\}$
then there exists a linear realization of the list
$L'=\{\ell_1^{a_1+k_1\ell_1}, \ell_2^{a_2+k_2\ell_2},\ldots,$ $\ell_n^{a_n+k_n\ell_n}\}$ for any $k_1,k_2,\ldots,k_n\in\mathbb{N}$.
Also, if the realization of $L$ is (almost) perfect, the realization of $L'$ has the same property.
\end{corollar}

\begin{proof}
Let $\r L$ be a linear realization of $L$. By Lemma \ref{l^l} there exists $\p\{\ell_i^{\ell_i}\}$ for any $\ell_i \in L$.
A linear realization of $L'$ is $\r L'=\sum_{i=1}^n k_i \cdot \p \{ \ell_i^{\ell_i}\}+\r L$,  see Lemma \ref{LemmaComposition} (4). If $\r L$ is almost perfect or perfect the thesis follows from Lemma
\ref{LemmaComposition} (2) and (1), respectively.
\end{proof}

\begin{ex}
Consider, as above, $L_1=\{1,3,6^ 2\}$ and $\p L_1$ corresponding to the sequence $S_3=(6_1 ,6_2,3,1,1,3,6_1,6_2,0)$.
Then, for instance, $\p \{1,2^ 4, 3^ 4,6^ 2 \}=2\cdot \p \{ 2^2 \}+\p\{ 3^ 3 \}+\p L_1$ exists and can be obtained taking the sequence
$(2_1,2_2,2_1,$ $2_2, 2_3,2_4,2_3,2_4, 3_1,3_2,3_3,3_1,3_2,3_3, 6_1,6_2,3_4,1,1,3_4,6_1,6_2,0)$.
\end{ex}

\section{First cases}

Let $L$ be a list of $n$ positive integers not exceeding $n$
and let $S$ be its underlying set.
We start investigating $\MPP(L)$  in the following cases: $|S|=1,2$ and
$n$.

\begin{rem}\label{inv}
Let $F$ be a near $1$-factor of $K_{2n+1}$ with $\ell(F)=\{x_1^{a_1},
x_2^{a_2},\ldots,x_{t}^{a_t}\}$. Let $y$ be an integer coprime with $2n+1$ and let $r_i$ be the remainder of the division of $yx_i$ by $2n+1$. Define $\ell_i=r_i$ if $1\leq r_i \leq n$ and $\ell_i=2n+1-r_i$ if $n<r_i\leq 2n$.
Then, multiplying each
vertex of $F$ by $y$, it is possible to obtain
a near $1$-factor $F'$ of $K_{2n+1}$ such that
$\ell(F')=\{\ell_1^{a_1}, \ell_2^{a_2}, \ldots, \ell_t^{a_t} \}$.
\end{rem}

If $|S|=1$, all the edges have the same length.
If this length is $1$ it is immediate
to see that $F=\{[2i, 2i+1]\mid i=0,\ldots,n-1\}\cup\{2n\}$ is
a near $1$-factor $F$ of $K_{2n+1}$ such that $\ell(F)=\{1^n\}$. Let now $S=\{x \}$ with $2\leq x \leq n$.
By Proposition \ref{nec}  we have to consider only the case $\gcd(x,2n+1)=1$.
By Remark \ref{inv}, multiplying the vertices
of $F$ by $x$, we obtain a near $1$-factor $F'$ of $K_{2n+1}$ such that
$\ell(F')=\{x^n\}$. Hence we can conclude that  $\MPP(\{x^n\})$ is true for any positive
integers $n$ and any
positive integer $x$ not exceeding $n$.

If $|S|=n$, namely if $L=S=\{1,2,3,\ldots,n\}$,
a near $1$-factor of $K_{2n+1}$ such that $\ell(F)=L$ is, up to translations,
a \emph{starter} of $\Z_{2n+1}$.
In fact, a starter in the odd order abelian group $G$ (written additively), where $|G|=2n+1$,
is a set of unordered pairs $R=\{\{r_i,t_i\}\mid 1\leq i\leq n\}$ that satisfies:
\begin{itemize}
\item[(1)] $\{r_i\mid  1\leq i\leq n\} \cup\{t_i \mid  1\leq i\leq n\}=G\setminus \{0\}$
\item[(2)] $\{\pm (r_i-t_i)\mid 1\leq i\leq n\}=G\setminus\{0\}.$
\end{itemize}
Hence $\MPP(\{1,2,3,\ldots,$ $n\})$ is always true for any positive integer $n$,
in fact it is sufficient to take $F=\{[i,2n-1-i]\ |\ i=0,\ldots,n-1\}\cup\{2n\}$, namely, the so-called \emph{patterned
starter} of $\Z_{2n+1}$,  see
\cite{D}.

So, we proved the following.

\begin{prop}\label{1n}
Let $L$ be a list of $n$ elements not exceeding $n$ with underlying set $S$. Then, $\MPP(L)$ holds if either $|S|=1$ or $|S|=n$.
\end{prop}

Now, we consider the case $|S|=2$, i.e. $L=\{x^a,y^b\}$, where $0<x<y\leq
a+b=n$. Observe that by Proposition \ref{nec} we have to consider only the case
$\gcd(x,y,2n+1)=1$. We start considering the case $x=1$.

\begin{lem}\label{1y}
There exists a linear realization of any list $L=\{1^a, y^b\}$ whenever $a\geq \lfloor
\frac{y-1}{2}\rfloor $.
\end{lem}

\begin{proof}
Set $v=2(a+b)+1$ and
write $b=qy+r$, where $0\leq
r < y$. Firstly, we consider the following sequences.
If $y-r$ is even, take
$$S_1=(y_1,y_2,\ldots, y_r, 1_1, 1_1,
1_2,1_2,\ldots,1_{\frac{y-r}{2}},1_{\frac{y-r}{2}},y_1,y_2,\ldots, y_r,0)$$
and if $y-r$ is odd, take
$$S_2=(y_1,y_2,\ldots, y_r, 1_1, 1_1,
1_2,1_2,\ldots,1_{\frac{y-r-1}{2}},1_{\frac{y-r-1}{2}},0,y_1,y_2,\ldots, y_r).$$
Clearly $S_1$ and $S_2$ are linear realizations of
$L'=\{1^{\lfloor\frac{y-r}{2}\rfloor},y^r\}$.
Also, observe that $S_1$ is perfect, while $S_2$ is
almost perfect if $r=1$.
Next we apply Lemma \ref{LemmaComposition}, obtaining $\r L=\p\{1^{a-\lfloor\frac{y-r}{2}\rfloor} \}+q\cdot
\p\{y^y\}+\r L'$.
\end{proof}

\begin{corollar}\label{Lemma12}
Given $L=\{1^a, 2^b\}$, then there exists a linear realization of $L$ for any $a,b\geq 1$.
\end{corollar}

\begin{proof}
We apply Lemma \ref{1y}, with $y=2$ and so $\lfloor \frac{y-1}{2}\rfloor=0$. Observe that
if $b$ is even, we have found a perfect realization of $L$. If $b$ is odd, our realization
is almost perfect.
\end{proof}

It follows from  previous corollary that, given a list $L=\{1^a, 2^b\}$ there exists a near
$1$-factor $F$ of $K_{v}$, where $v=2(a+b)+1$, such that $\delta(F)=L$ whose isolated vertex
is $v-2$ if $b$ is odd, $v-1$ if $b$ is even. In the following, we will denote by $F+g$, $g\in \N$,
the graph obtained from $F$ by replacing each vertex $x$ of $F$ with $x+g$.
Clearly, $F+g$ is a near $1$-factor of the complete graph whose vertex set is $\{g,1+g,\ldots, 2(a+b)+g\}$ such that $\delta(F+g)=L$. This remark will be very useful in the next section.

Now we are ready to deal with the general case of two distinct edge lengths.
 \begin{prop}\label{2lunghezze}
$\MPP(\{x^a,y^b\})$ holds for any $a,b,x,y>0$.
\end{prop}

\begin{proof}
Let  $L=\{x^a,y^b\}$ with $x,y\leq a+b$ and let $v=2(a+b)+1$. We have to split the proof in four cases.
But firstly it is important to
observe that if $a\geq b$ then $a\geq \lfloor\frac{y-1}{2}\rfloor$, since $y\leq a+b$.
\medskip

\noindent Case 1.  $\gcd(x,v)=\gcd(y,v)=1$.
If $a\geq b$ we multiply the elements of $L$ by $x^{-1}$ and if $a< b$, we multiply by $y^{-1}$.
In both
cases, we obtain a list $L'$ with underlying set $S'=\{1,z\}$ and $z\leq a+b$. So, $L'$ satisfies the assumption of Lemma \ref{1y} and hence, there exists a near $1$-factor $F'$ such that $\ell(F')=L'$. By  Remark \ref{inv}, there exists a near $1$-factor $F$ of $K_v$
such that $\ell(F)=L$.
\medskip

\noindent Case 2. $\gcd(v,y)=d_2>1$ and $a\leq \frac{(d_2-1)v}{2d_2}$.
Clearly it has to be $\gcd(x,d_2)=1$. Also by Proposition \ref{nec}, $a\geq
\frac{d_2-1}{2}$.

We write the vertices of $K_v$, namely the integers $0,1,\ldots,v-1$, within a matrix $M$ of
$d_2$ rows and $\frac{v}{d_2}$ columns, where the element $m_{i,j}$ is
$(i-1)x+(j-1)y$ modulo $v$, for $i=1,\ldots,d_2$ and $j=1,\ldots,\frac{v}{d_2}$.
It is not hard to see that all the integers $m_{i,j}$ are distinct.
In fact if, by way of contradiction,  $m_{i,j}=m_{h,k}$,
then $(i-1)x+(j-1)y \equiv (h-1)x+(k-1)y\pmod v$. In particular, this means that $i
x\equiv hx \pmod{d_2}$, namely $d_2 \mid (i-h) x $.
Since $\gcd(x,d_2)=1$, it follows that $ d_2 \mid (i-h)$ and so
$i=h$, since $i,h\leq d_2$. Hence, $(j-1)y \equiv (k-1)y \pmod v$. Then, $v\mid (j-k)y$ and so
$\frac{v}{d_2}\mid (j-k) \frac{y}{d_2}$. We obtain that $\frac{v}{d_2} \mid (j-k)$, i.e.
$j=k$, since $j,k\leq \frac{v}{d_2}$.

Now, we construct the edges for the required near $1$-factor using exactly once all but one the elements
of the matrix $M$ as vertices.
Consider the following   $\frac{d_2-1}{2}$ edges of length $x$:
$[m_{2i+1,1},m_{2i+2,1}]$ for $i=0,\ldots,\frac{d_2-3}{2}$.

If $a-\frac{d_2-1}{2}$ is even, let $a-\frac{d_2-1}{2}=(\frac{v}{d_2}-1)q+r$ with $0\leq r< \frac{v}{d_2}-1$, note that $r$ is even.
We take the following edges of length $x$: $[m_{2i+1,j},m_{2i+2,j}]$ for $i=0,\ldots,q-1$
and $j=2,\ldots,\frac{v}{d_2}$ (if $q>0$) and $[m_{2q+1,j},m_{2q+2,j}]$ for $j=2,\ldots,r+1$.
It is easy to see that, since $r$ is even, there are an even number of elements
which are not used to construct the above edges, in each row of $M$ except for
the last one. Hence we construct the $b$ edges of length $y$ connecting pairs of adjacent
elements in each rows.
Since $a\leq \frac{(d_2-1)v}{2d_2}$ we have a sufficient number of edges of length $y$.
 In such a way we obtain a near $1$-factor of $K_v$ whose isolated vertex is
$m_{d_2,1}$.

If $a-\frac{d_2-1}{2}$ is odd, we construct the following edge of length $x$.
Observing that $m_{d_2,1}+x=(d_2-1)x+x=d_2x$ is an element of the first row, say $m_{1,k}$, we can take
the edge $[m_{d_2,1},m_{1,k}]$.
Hence we have to construct other $a-\frac{d_2+1}{2}$ edges of length $x$.
Since $a-\frac{d_2+1}{2}$ is even,
 let $a-\frac{d_2+1}{2}=(\frac{v}{d_2}-1)q+r$ with $0\leq r< \frac{v}{d_2}-1$, note that $r$ is even.
We take the following edges of length $x$: $[m_{2i,j},m_{2i+1,j}]$ for $i=1,\ldots,q$
and $j=2,\ldots,\frac{v}{d_2}$ (if $q>0$) and $[m_{2q+2,j},m_{2q+3,j}]$ for $j=2,\ldots,r+1$.
It is easy to see that, since $r$ is even, there are an even number of elements,
which are not used to construct the above edges, in each row of $M$ except for
the first one.
Hence, we can construct the edges of length $y$ as in the previous case.
\medskip

\noindent Case 3. $\gcd(v,y)=d_2>1$, $a> \frac{(d_2-1)v}{2d_2}$, $\gcd(x,v)=1$.
By the assumption on $a$, it follows that $a> b$. In fact we have
$a>\frac{(d_2-1)v}{2d_2}=\frac{d_2-1}{d_2}\frac{v}{2}\geq
\frac{2}{3}\frac{v}{2}=\frac{v}{3}= \frac{2(a+b)+1}{3}$ and so
$a> 2b+1>b$. Hence we can proceed analogously to Case 1, applying Remark \ref{inv} and Lemma \ref{1y}.
\medskip

\noindent Case 4. $\gcd(v,y)=d_2>1$, $a> \frac{(d_2-1)v}{2d_2}$, $\gcd(x,v)=d_1>1$.
Note that in this case $b \leq
\frac{(d_1-1)v}{2d_1}$. In fact, if, by way of contradiction,
we have $b >
\frac{(d_1-1)v}{2d_1}$ it results

\begin{small}
$$
\frac{v-1}{2} = a+b > \frac{(d_1-1)v}{2d_1} +\frac{(d_2-1)v}{2d_2}
  =  \frac{v}{2} \big( \frac{d_1-1}{d_1} + \frac{d_2-1}{d_2} \big)\\
  \geq   \frac{v}{2}
\big(\frac{2}{3}+\frac{2}{3}\big)=\frac{2v}{3} > \frac{v-1}{2},
$$
\end{small}

\noindent which clearly is a contradiction.
Hence we can apply the same process of Case 2, interchanging $x$ with $y$.
\end{proof}

\begin{ex}
Take now $L=\{6^9,10^{13}\}$. Here $v=45$,
$d_1=\gcd(6,45)=3$, $d_2=\gcd(10,45)=5$.
Note that we are in Case 2 of Proposition \ref{2lunghezze},
and that $a-\frac{d_2-1}{2}=9-2=7$ is odd.
Firstly, we construct the matrix
$$M=\left(\begin{array}{ccccccccc}
 {\bf0} &  10 &  20 &  {\bf30} &  40 &  5 &  15 & 25 & \emph{35}\\
  {\bf6} & {\bf16} & {\bf26} & {\bf36} & {\bf1} & {\bf11} & {\bf21} & 31 & 41 \\
  {\bf12} & {\bf22} & {\bf32} & {\bf42} & {\bf7} & {\bf17} & {\bf27} & 37 & 2 \\
{\bf18} & 28 & 38 & 3 & 13 & 23 & 33 & 43 & 8 \\
 {\bf24} & 34 & 44 & 9 & 19 & 29 & 39 & 4 & 14 \\
  \end{array} \right).$$
\end{ex}

Reasoning as in the proof of previous proposition we obtain the following edges of lengths $6$:
$[0,6], [12,18], [24,30], [16,22], [26,32], [36,42], [1,7],[11,17],$ $[21,27]$.
The elements used to construct these edges are highlighted in bold in the matrix $M$.
Now, we can construct the following edges of length $10$:
$[10,20], [40,5],[15,25],[31,$ $41], [37,2],[28,38], [3,13],[23,33],[43,8],[34,44], [9,19], [29,39], [4, 14]$. So, the isolated vertex is $35$.

\section{Near $1$-factors with edge lengths $1,2,t$}

In this section we investigate $\MPP(L)$ where $L=\{1^a,2^b,t^c\}$
for any integer $t>2$. Clearly, in view of the results of the previous section,
we can assume $a,b,c\geq 1$.

\begin{prop}\label{big}
Let $L=\{1^a,2^b,t^c\}$ with  $a+b\geq \lfloor\frac{t-1}{2}\rfloor$. Then there exists
a linear realization of $L$.
\end{prop}

\begin{proof}
Let $c=tq+r$ with $0\leq r\leq t-1$. If $r=0$ the statement follows from Lemmas \ref{LemmaComposition} and \ref{l^l} and
Corollary \ref{Lemma12}.
So we can assume $r\geq 1$.
Since  $a+b\geq \lfloor\frac{t-1}{2}\rfloor$ then $a+b\geq \lfloor\frac{t-r}{2}\rfloor$.
We start constructing a linear realization of $L'=\{1^a,2^b,t^r\}$ and
we have to distinguish four cases according to the congruence class of $t-r$ modulo $4$.
\medskip

\noindent Case 1. $t-r\equiv 0 \pmod 4$.\\
If  $b\geq \frac{t-r}{2}$, we consider the sequence
$S=(t_1,t_2,\ldots,t_r,2_1,2_2,2_1,2_2,\ldots,2_{\frac{t-r}{2}-1},$ $ 2_{\frac{t-r}{2}}, 2_{\frac{t-r}{2}-1}, 2_{\frac{t-r}{2}},t_1,t_2,
\ldots ,t_r,0)$ which is a perfect realization of $L''$ $=\{2^{\frac{t-r}{2}},t^r\}$. Next, consider
$\r L'=\p L''+\r\{1^a, 2^{b-\frac{t-r}{2}}\}$.\\
Now, assume $b< \frac{t-r}{2}$. If $b$ is even, the sequence
$S=(t_1,t_2,\ldots ,t_r,2_1,2_2,2_1,$ $2_2,\ldots,2_{b-1},2_{b},
2_{b-1},2_{b},1_1,1_1,\ldots,1_{\frac{t-r-2b}{2}}, 1_{\frac{t-r-2b}{2}},t_1, t_2,\ldots
,t_r,0 )$ is a perfect linear realization of $L''=\{1^{\frac{t-r-2b}{2}}, 2^b,t^r\}$.
In this case $\p L'=\p L''+\p\{1^{a-\frac{t-r-2b}{2}}\}$.
If $b$ is odd, the sequence
$S=(t_1,t_2,\ldots ,t_r,2_1,2_2,2_1,2_2,\ldots, 2_{b-2},2_{b-1},
2_{b-2}, 2_{b-1},1_1,$ $1_1,\ldots, 1_{\frac{t-r-2b}{2}+1},1_{\frac{t-r-2b}{2}+1},t_1,t_2,
\ldots ,t_r,0)$
is a perfect linear realization of $L''=\{1^{\frac{t-r-2b}{2}+1} ,$ $2^{b-1},t^r\}$.
Hence, $\ap L'=\p L''+\ap\{1^{a-\frac{t-r-2b}{2}-1},2\}$.

\medskip
\noindent Case 2. $t-r\equiv 1 \pmod 4$.\\
Suppose $b$ even. If $b\geq \frac{t-r-1}{2}$ the sequence
$S=(t_1,t_2,\ldots,t_r,2_1, 2_2,2_1,2_2,\ldots, 2_{\frac{t-r-3}{2}},$ $2_{\frac{t-r-1}{2}},2_{\frac{t-r-3}{2}},2_{\frac{t-r-1}{2}},0,t_1,t_2,\ldots,t_r)$
is a linear realization of $L''=\{2^{\frac{t-r-1}{2}},$ $t^r\}$ and
$\r L'= \p\{1^a, 2^{b-\frac{t-r-1}{2}}\}+\r L''$.
If  $b <  \frac{t-r-1}{2}$, the sequence
$S=(t_1,t_2,\ldots,t_r,2_1,2_2,$ $2_1,2_2,\ldots,
2_{b-1},2_{b}, 2_{b-1},2_{b},
1_1,1_1,\ldots,1_{\frac{t-r-2b-1}{2}}, 1_{\frac{t-r-2b-1}{2}}, 0,t_1,t_2,\ldots,t_r)$
is a linear realization of $L''=\{1^{\frac{t-r-2b-1}{2}},2^b,t^r\}$. In this case,
$\r L'=\p\{1^{a-\frac{t-r-2b-1}{2}}\}+\r L''$.\\
Suppose now $b$ odd. If $b\geq \frac{t-r-1}{2}$ the sequence
$S=(t_1,t_2,\ldots,t_r,1_1,1_1,2_1,2_2,$ $2_1,2_2,\ldots,$ $2_{\frac{t-r-3}{2}},0,2_{\frac{
t-r-3}{2}},t_1,t_2,\ldots,t_r)$
is a linear realization of  $L''=\{1,$ $2^{\frac{t-r-3}{2}},$ $t^r\}$. Hence $\r L'=\p\{1^{a-1},2^{b-\frac{
t-r-3}{2}}\}+\r L''$. If $b <  \frac{t-r-1}{2}$ the sequence
$S=(t_1,t_2,\ldots,t_r,2_1,2_2,2_1,$ $2_2,\ldots,
2_b,0,2_b,
1_1, 1_1,\ldots,1_{\frac{t-r-2b-1}{2}},1_{\frac{t-r-2b-1}{2}},t_1, t_2,\ldots,
t_r)$
is a linear realization of  $L''=\{1^{\frac{t-r-2b-1}{2}},2^b,t^r\}$. In this case, we obtain
$\r L= \p\{1^{a-\frac{t-r-2b-1}{2}}\}+\r L''$.
\medskip

\noindent Case 3. $t-r\equiv 2 \pmod 4$.\\
If $b\geq \frac{t-r-2}{2}$, the sequence
$S=(t_1,t_2,\ldots ,t_r,1_1,1_1,2_1,2_2,2_1,2_2,\ldots, 2_{\frac{t-r}{2}-2},$
$2_{\frac{t-r}{2}-1},$ $2_{\frac{t-r}{2}-2},2_{\frac{t-r}{2}-1},t_1,t_2,\ldots
,t_r,0)$
is a perfect linear realization of $L''=\{1,2^{\frac{t-r}{2}-1} ,t^r\}$. Hence,
$\r L=\p L''+\r\{1^{a-1}, 2^{b-\frac{t-r}{2}+1}\}$.\\
Assume now $b< \frac{t-r-2}{2}$. If $b$ is even,
we take
the sequence
$S=(t_1,t_2,\ldots ,t_r,2_1,$ $2_2,2_1,2_2,$ $\ldots,2_{b-1},2_b,2_{b-1},2_b,
1_1,1_1,\ldots,1_{\frac{t-r-2b}{2}},1_{\frac{t-r-2b}{2}},$ $t_1,t_2,\ldots ,t_r,0)$, that
is a perfect linear realization of $L''=\{1^{\frac{t-r-2b}{2}},2^b,t^r\}$. In this case,
$\p L'=\p L''+\p\{1^{a-\frac{t-r-2b}{2}}\}$. If $b$ is odd, we take the
sequence
$S=(t_1,t_2,\ldots ,t_r,2_1,2_2,2_1,2_2,$ $\ldots,
2_{b-2},2_{b-1},2_{b-2},2_{b-1},
1_1,1_1,\ldots,1_{\frac{t-r-2b}{2}+1},1_{\frac{t-r-2b}{2}+1},$ $t_1,t_2,\ldots ,t_r,0)$,
that is  a  perfect linear realization of $L'=\{1^{\frac{t-r-2b}{2}+1} ,2^{b-1},t^r\}$.
So, $\ap L'=\p L''+ \ap\{1^{a-\frac{t-r-2b}{2}-1},$  $2\}$.

\medskip
\noindent Case 4. $t-r\equiv 3 \pmod 4$.\\
Suppose $b$  even. If $b\geq \frac{t-r-1}{2}$ the sequence
$S=(t_1,t_2,\ldots,t_r,1_1,1_1,2_1,2_2,2_1,2_2,$ $\ldots,
2_{\frac{t-r-5}{2}},
2_{\frac{t-r-3}{2}},2_{\frac{t-r-5}{2}},2_{\frac{t-r-3}{2}},
0,t_1,t_2,\ldots,t_r)$
is a linear realization of  $L''=\{1,2^{\frac{t-r-3}{2}} ,$ $t^r\}$. In this case,
$\r L'=\p\{1^{a-1},2^{b-\frac{t-r-3}{2}}\}+\r L''$.
If  $b <  \frac{t-r-1}{2}$ the sequence
$S=(t_1,t_2,\ldots,t_r,2_1,2_2,2_1,2_2,\ldots,2_{b-1},2_b,$
$2_{b-1},2_b,1_1,1_1,\ldots,1_{\frac{t-r-2b-1}{2}},$
$1_{\frac{t-r-2b-1}{2}},0,t_1, t_2,\ldots,t_r)$
is a linear realization of  $L''=\{1^{\frac{t-r-2b-1}{2}} ,2^b,t^r\}$. So, we have
$\r L'=\p\{1^{a-{\frac{t-r-2b-1}{2}}}\}+\r L''$. \\
Suppose now $b$ odd.  If $b\geq \frac{t-r-1}{2}$ the sequence
$S=(t_1,t_2,\ldots,t_r,2_1,2_2,2_1,2_2,$ $\ldots, 2_{\frac{t-r-1}{2}},
0,2_{\frac{t-r-1}{2}},t_1,t_2,\ldots,t_r)$
is a linear realization of  $L''=\{2^{\frac{t-r-1}{2}},t^r\}$. Here,
$\r L'=p\{1^a, 2^{b-\frac{t-r-1}{2}}\}+\r L''$.
If  $b <  \frac{t-r-1}{2}$ the sequence
$S=(t_1,t_2,\ldots,t_r,$ $2_1,2_2,2_1,2_2,\ldots, 2_{b},0,2_{b},1_1, 1_1,\ldots,1_{\frac{t-r-2b-1}{2}},1_{\frac{t-r-2b-1}{2}},
t_1,t_2,\ldots,t_r)$
is a linear realization of  $L''=\{1^{\frac{t-r-2b-1}{2}},2^b,t^r\}$. Hence,
$\r L'=\p\{1^{a-\frac{t-r-2b-1}{2}}\}+\r L''$ is a linear realization
 of  $L'=\{1^a,2^b,t^r\}$.\\
Now in order to obtain the thesis it is sufficient to
note that $\r L=q\cdot \p\{t^t\}+\r L'$.
\end{proof}

\begin{ex}
Let $L=\{1^4,2^2,12^{26}\}$. Since $4+2> \lfloor\frac{12-1}{2}\rfloor$
we are in the hypothesis of previous proposition. Note that in this case
$q=2$ and $r=2$, hence $t-r=12-2\equiv 2\pmod 4$ so we are in Case 3.
Also $b=2<4=\frac{t-r-2}{2}$, so, since $b$ is even, we take the following perfect linear realization of
$L'=\{1^3,2^2,12^2\}$:
$S=(12_1,12_2,2_1,2_2,2_1,2_2,1_1,1_1,1_2,1_2,1_3,1_3,12_1,12_1,0)$.
Now a perfect linear realization of $L$ is given by $2\cdot \p\{12^{12}\}+\p\{1^1\}+\p L'$.
\end{ex}

\begin{prop}\label{12t-notcop}
$\MPP(L)$ holds for any $L=\{1^a,2^b,t^c\}$,
with $\gcd(t,v)=d>1$, where $v=2(a+b+c)+1$.
\end{prop}

\begin{proof}
Clearly, $t\leq \frac{v-1}{2}$ and
by Proposition \ref{nec} we can assume $a+b\geq \frac{d-1}{2}$.
Also, in view  of Proposition \ref{big} we can assume $a+b<\lfloor\frac{t-1}{2}\rfloor$ which implies
$c>\frac{v-1}{2}-\lfloor\frac{t-1}{2}\rfloor$, hence we have also $c\geq\frac{v+d}{2d}$.
In fact, if, by way of contradiction, $\frac{v-1}{2}-\lfloor\frac{t-1}{2}\rfloor<\frac{v+d}{2d}$
we have $\frac{v-1}{2}-\frac{v+d}{2d}<\lfloor\frac{t-1}{2}\rfloor\leq \frac{v-3}{4}$, since $t\leq \frac{v-1}{2}$.
Namely $\frac{v-1}{2}<\frac{v}{d}$, which cannot be since $d\geq 3$.\\
Now, let $M=(m_{i,j})$ be the $d\times \frac{v}{d}$ matrix whose elements are so defined, for
$i=1,\ldots,d$ and $j=1,\ldots,\frac{v}{d}$:
$$m_{i,j}=\left\{\begin{array}{ll}
(i-1)+(j-1)t\pmod v &  \textrm{ if } i\equiv 0,1\pmod 4\\
i+(j-1)t\quad \pmod v &  \textrm{ if } i\equiv 2 \pmod 4\\
(i-2)+(j-1)t \pmod v  \quad &  \textrm{ if }  i\equiv 3 \pmod 4.
\end{array}\right.$$
It is easy to see that the elements of $M$ are the vertices of $K_v$.
Also, it is not hard to check that $M$ has at least $5$ columns. In fact, $\frac{v}{d}$ is an odd number and $\frac{v}{d}\neq3$.
Indeed if $\frac{v}{d}=3$ we would have $t=d$ and hence $a+b\geq \frac{d-1}{2}=\frac{t-1}{2}$ but we are assuming $a+b<\lfloor\frac{t-1}{2}\rfloor$.
Note that, for any $i$ and $j$,  $m_{i,j+1}-m_{i,j}=t \pmod v$, namely two elements adjacent on a row form an edge of length $t$.
Also $|m_{i+1,j}-m_{i,j}|=2$ if $i$ is odd, while $|m_{i+1,j}-m_{i,j}|=1$ if $i$ is even.
Our aim is to construct $a$ edges of length $1$ and $b$ edges of length $2$ in such a way that the elements of $M$
not used to obtain these edges appear in all but one  rows (say $R$) as an even number of adjacent entries.
In fact, if this holds, then it will be immediate to construct the edges of length $t$ with pairs of elements adjacent on a row.
Obviously, the isolated vertex of the near $1$-factor will belong to the row $R$.
So, it suffices to explain how to construct the edges of length $1$ and $2$. Then, the reader can easily check that the previous condition is always satisfied in the following constructions.
We have to split the proof into two cases.

\medskip
\noindent Case 1.  $a+b-\frac{d-1}{2}$ even.\\
By the assumptions on $a,b$ and $d$ it is possible to find two positive integers
$\tilde a$ and $\tilde b$ such that $\tilde a+\tilde b=\frac{d-1}{2}$, $a-\tilde a=2\alpha$
and $b-\tilde b=2\beta$, for suitable $\alpha$ and $\beta$.
Clearly, in general, it is possible more than one choice for the pair $(\tilde a,\tilde b)$.
By Corollary \ref{Lemma12} there exists a \nf $F$ of $K_d$ such that $\delta(F)=\{1^{\tilde a}, 2^{\tilde b}\}$. We start taking the edges of $F$.
Now we have to construct $2\alpha$ edges of length $1$ and $2\beta$ edges of length $2$.
Let $2\alpha=(\frac{v}{d}-1)\bar q+\bar r$ with $0\leq \bar r< \frac{v}{d}-1$, note that $\bar r$ is even.
Take the following
$2\alpha$ edges of length $1$:
$[m_{2i,j},m_{2i+1,j}]$ for $i=\frac{d-1}{2}-\bar q+1,\ldots,\frac{d-1}{2}$, $j=2,\ldots,\frac{v}{d}$
(if $\bar q>0$)
and $[m_{2(\frac{d-1}{2}-\bar q),j},m_{2(\frac{d-1}{2}-\bar q)+1,j}]$ for $j=2,\ldots, \bar r +1$.
Let $2\beta=(\frac{v}{d}-1)\tilde q+\tilde r$ with $0\leq \tilde r< \frac{v}{d}-1$, note that $\tilde r$ is even.
Take the following
$2\beta$ edges of length $2$: $[m_{2i+1,j},m_{2i+2,j}]$ for $i=0,\ldots,\tilde q-1$, $j=2,\ldots, \frac{v}{d}$
(if $\tilde q>0$)
and $[m_{2\tilde q+1,j},m_{2\tilde q+2, j}]$ for $j=\frac{v}{d}-\tilde r+1,\ldots,\frac{v}{d}$.\\
Since  $c\geq \frac{v+d}{2d}$, the edges so constructed involve distinct vertices of $K_v$.
Also, since $\bar r$ and $\tilde r$ are even one can check that the elements of $M$
not used to obtain these edges appear in all but one  rows as an even number of adjacent entries.
\medskip

\noindent Case 2. $a+b-\frac{d-1}{2}$ odd.\\
By the assumptions on $a,b$ and $d$ it is possible to find two positive integers
$\tilde a$ and $\tilde b$ such that $\tilde a+\tilde b=\frac{d-1}{2}$, $a-\tilde a=2\alpha+1$
and $b-\tilde b=2\beta$, for suitable $\alpha$ and $\beta$.
Also in this case,  it may be possible more than one choice for the pair $(\tilde a,\tilde b)$.
From Corollary \ref{Lemma12} there exists a near $1$-factor $F$ of $K_d$ such that
$\delta(F)=\{1^{\tilde a},2^{\tilde b}\}$ whose isolated vertex is in $\{d-2,d-1\}$, namely in $\{m_{d-1,1},m_{d,1}\}$.
We start taking the edges of $F$. Next, we take the edges of length $2$.
Let $2\beta =(\frac{v}{d}-1)\bar q + \bar r$ with $0\leq \bar r <  \frac{v}{d}-1$, note that $\bar r$ is even.
We consider the edges $[m_{2i+1,j},m_{2i+2,j}]$, for $i=0,\ldots,\bar q -1$ and $j=2,\ldots,\frac{v}{d}$
(if $\bar q>0$),
and the edges $[m_{2\bar q+1,j},m_{2\bar q+2,j}]$,  for $j=\frac{v}{d}-\bar r+1,\ldots,\frac{v}{d}$.\\
Now, we take the edges of length $1$.
Firstly we take the edge $[m_{d-1,3},m_{d,3}]$.
Next we distinguish two cases.
If $2\alpha\leq \frac{v}{d}-3$ we take the edges $[m_{d-1,j},m_{d,j}]$ for $j=4,\ldots, 2\alpha+3$.
If $2\alpha> \frac{v}{d}-3$,
let $2\alpha-(\frac{v}{d}-3) =(\frac{v}{d}-1)\tilde q + \tilde r$ with $0\leq \tilde r < \frac{v}{d}-1$, note that $\tilde r$ is even.
We  consider the edges
$[m_{d-1,j},m_{d,j}]$ for $j=4,\ldots, \frac{v}{d}$,
 $[m_{d-2i,j},m_{d-2i-1,j}]$,
for $i=1,\ldots,\tilde q$ and $j=2,\ldots,\frac{v}{d}$ (if $\tilde q>0$),
and  $[m_{d-2\tilde q-2,j},m_{d-2\tilde q-3,j}]$,
for $j=2,\ldots,\tilde r+1$.
Since  $c\geq \frac{v+d}{2d}$, the edges so constructed involved distinct vertices of $K_v$.
The isolated vertex is in $\{m_{d-1,2},m_{d,2}\}$.
\end{proof}

\begin{ex}
Let $L=\{1^{3},2^{2},21^{19}\}$. Hence, we have $v=49$, $d=7$ and $a+b-\frac{d-1}{2}=3+2-3=2$ is even.
So we are in Case 1 of Proposition \ref{12t-notcop}.
We start constructing the $7 \times 7$ matrix $M$:
$$M=\left(\begin{array}{ccccccc}
\emph{0} &  21 &  42 & 14 & 35 &  {\bf 7} & {\bf 28} \\
 {\bf 2} & 23 & 44  & 16 & 37 &  {\bf 9} & {\bf 30}\\
  {\bf 1} & 22 & 43 & 15 & 36 & 8 & 29\\
{\bf 3} & 24 & 45  & 17 & 38 & 10 & 31\\
{\bf 4} & 25 & 46 & 18 & 39 & 11 & 32\\
{\bf 6} & 27 & 48 & 20 & 41 & 13 & 34\\
{\bf 5} & 26 & 47 & 19 & 40 & 12 & 33
  \end{array} \right).$$
  Note that in this case we have more than one choice for $\tilde a$ and $\tilde b$, in fact we can choose
$(\tilde a, \tilde b)=(3,0)$ or $(1,2)$. In this example we choose $\tilde a=3$ and $\tilde b=0$.
So we take $F=\{[2,1], [3,4],[6,5]\}\cup\{0\}$.
Now we have $a-\tilde a=0$ and $b-\tilde b=2$, hence $\alpha =0$ and $\beta=1$.
So we have to take the following edges of length $2$: $[7,9]$ and
$[28,30]$.
The elements used up to now to construct the edges of length $1$ and $2$ are highlighted  in bold in the matrix.
Finally, it is easy to construct $19$ edges of length $21$ as follows: $[21,42],[14,35],[23,44],[16,37],[22,43],
[15,36],[8,29],[24,$ $45],[17,38],[10,31],[25,46],[18,39],[11,32],[27,48],[20,41],[13,34],[26,47],[19, 40],$ $[12,33]$.
Clearly, the isolated vertex is $0$.
\end{ex}

\begin{ex}
Let $L=\{1^{7},2^{4},25^{16}\}$. Hence, we have $v=55$, $d=5$ and $a+b-\frac{d-1}{2}=7+4-2=9$ is odd.
So we are in Case 2 of Proposition \ref{12t-notcop}.
We start constructing the $5 \times 11$ matrix $M$:
$$M=\left(\begin{array}{ccccccccccc}
{\bf 0} &  25 &  50 & 20 & 45 & 15 & 40 & 10 & 35 & {\bf5} & {\bf30} \\
 {\bf 2} & 27 & 52  & 22 & 47 & 17 & 42 & 12 & 37 & {\bf7} & {\bf32}\\
  {\bf 1} & 26 & 51 & 21 & 46 & 16 & 41 & 11 & 36 & 6 & 31\\
{\bf 3} & \emph{28} & {\bf53}  & {\bf23} & {\bf48} & {\bf18} & {\bf43} & {\bf13} & {\bf38} & 8 & 33\\
4 & 29 & {\bf54} &  {\bf24} & {\bf49} & {\bf19} & {\bf44} & {\bf14} & {\bf39} & 9 & 34
  \end{array} \right).$$
\noindent
Note that in this case we have more than one choice for $\tilde a$ and $\tilde b$, in fact we can choose
$(\tilde a, \tilde b)=(0,2)$ or $(2,0)$. In this example we choose $\tilde a=0$ and $\tilde b=2$.
So we take $F=\{[0,2], [1,3]\}\cup\{4\}$.
Now we have $a-\tilde a=7$ and $b-\tilde b=2$, hence $\alpha =3$ and $\beta=1$.
So we have to take the following edges of length $2$: $[5,7]$ and
$[30,32]$.
Take now $[m_{4,3},m_{5,3}]=[53,54]$.
  Also, take the following edges of length $1$: $[23,24],[48,49],[18,19],[43,44],[13,14],[38,39]$.
The elements used up to now to construct the edges of length $1$ and $2$ are highlighted  in bold in the matrix.
Finally, it is easy to construct $16$ edges of length
$25$ as follows: $[25,50],[20,45],[15,40],[10,35],[27,52],[22,47],[17,42],[12,37],[26,51],[21, 46], [16,$ $41],[11,36],[6,31],[8,33],[4,29],[9,34]$.
Clearly, the isolated vertex is $28$.
\end{ex}

In order to complete the study of $\MPP(\{1^a, 2^ b, t^ c\})$ we are left to consider the case
$\gcd(t,v)=1$ and $a+b<\lfloor\frac{t-1}{2}\rfloor$, where $v=2(a+b+c)+1$.
In the following  we give some partial results about this case.
So we take $a+b<\lfloor\frac{t-1}{2}\rfloor$, but not ``too small''.
Clearly,  we believe that $\MPP$ conjecture holds also for $a+b$ ``very small''
 and we will give some examples for these cases at the end of this section.

\begin{prop}\label{12t-cop}
Let $L=\{1^a,2^b,t^c\}$, where $t$ is an integer with $t>2$. Let $v=2(a+b+c)+1$ and
let $r'$ be the remainder of the division of $v$ by $2t$.
Assume that $a+b\geq \frac{r'-1}{2}$ if $\lfloor\frac{v}{t}\rfloor$ is even
and that $a+b\geq t- \frac{r'+1}{2}$ if $\lfloor\frac{v}{t}\rfloor$ is odd.
Then there exists a near $1$-factor $F$ of $K_v$
such that $\ell(F)=L$.
\end{prop}

\begin{proof}
In view of Propositions \ref{big} and \ref{12t-notcop}, we may assume $a+b< \lfloor \frac{t-1}{2}\rfloor$
and  $\gcd(t,v)=1$.
Let $v=qt+r$, with $0<r<t$. We construct an incomplete matrix $M$ with $t$ rows and
$q+1$ columns whose elements are $\{0,1,2,\ldots,v-1\}$, namely the vertices of $K_v$, and the element
 on the $i$-th row and $j$-th column
is $m_{i,j}=(i-1)+(j-1)t$, for $i=1,\ldots,t$, $j=1,\ldots,q$ and $i=1,\ldots,r$ if $j=q+1$.
Hence the first $q$ columns are complete, while in the last column we have
only $r$ elements. It is easy to see that $m_{i,j+1}-m_{i,j}=t$ for any $i,j$, hence two elements adjacent on a row
can be connected by an edge of length $t$.
Also, we have $m_{i+1,j}-m_{i,j}=1$ for any $i,j$, so
two elements adjacent on a column can be connected by an edge of length $1$.

As in Proposition \ref{12t-notcop},
our aim is to construct $a$ edges of length $1$ and $b$ edges of length $2$ in such a way that the elements of $M$
not used to obtain these edges appear in all but one rows as an even number of adjacent entries.
So, in the following, we will explain how to construct the edges of length $1$ and $2$ in such a way that the
previous condition is always satisfied.

Since $v$ is odd and $r'$ is the remainder of the division of $v$ by $2t$,  $r'$ is an odd integer too.
Now we split the construction into two cases.
\medskip

\noindent Case 1. $q$ even.\\
Note that,  in this case, the  rows with an odd number of elements
are the first $r'=r$ rows of $M$.\\[2pt]
Case 1A. $a+b=\frac{r'-1}{2}$.\\
By Corollary \ref{Lemma12}, there exists a near $1$-factor $F$ of $K_{r'}$ with $\delta(F)=\{1^a,2^b\}$.
Then it suffices to take the edges of $F$.\\[2pt]
Case 1B. $a+b>\frac{r'-1}{2}$.\\
Firstly, consider the case $r'=1$. We apply Corollary \ref{Lemma12} to obtain a near $1$-factor $\tilde F$ such that
$\delta(\tilde F)=L'=\{1^{\lfloor\frac{a}{2} \rfloor}, 2^{\lfloor\frac{b}{2} \rfloor}\}$.
Next, in $K_v$ we take the edges of the graphs
$\tilde F+t-2|L'|-1$ and $\tilde F+2t-2|L'|-1$.
If $a$ and $b$ are both even, we have done. If $a$ is even and $b$ is odd, we still have to construct an edge of length $2$, take $[1,v-1]$.
If $a$ is odd and $b$ is even, we still have to construct an edge of length $1$, take $[0,v-1]$.
Finally, if $a$ and $b$ are both odd,  we still have to construct an edge of length $1$ and one of length $2$,
take $[v-t-1,v-t]$ and $[1,v-1]$.

Assume now $r'>1$. The construction in this subcase depends on the parity of $a+b-\frac{r'-1}{2}$.
It is easy to see that if $a+b-\frac{r'-1}{2}$ is even
there exist $\tilde a$, $\tilde b$ such that
$\tilde a+\tilde b=\frac{r'-1}{2}$ and $a-\tilde a=2\alpha$ and $b-\tilde b=2\beta$, for suitable $\alpha$ and $\beta$.
Now,
it is possible, in view of Corollary \ref{Lemma12},
to construct a near $1$-factor $F$ of $K_{r'}$ such that  $\delta(F)=\{1^{\tilde a},2^{\tilde b}\}$
and a \nf  $F'$ of $K_{2(\alpha+\beta)+1}$ such that $\delta(F')=\{1^\alpha,2^\beta\}$. Note that, since $a+b < \lfloor\frac{t-1}{2}\rfloor$, we have $2(\alpha+\beta)+1<t$.
Now, in $K_v$ we take the edges of $F+tq$, $F'+t-2(\alpha+\beta)-1$ and $F'+2t-2(\alpha+\beta)-1$,
these are exactly $a$ edges of length $1$ and $b$ edges of length $2$.

It is easy to see that if $a+b-\frac{r'-1}{2}$ is odd,
there exist $\tilde a$, $\tilde b$ such that
$\tilde a+\tilde b=\frac{r'-1}{2}$ and $a-\tilde a=2\alpha+1$ and $b-\tilde b=2\beta$, for suitable $\alpha$ and $\beta$.
Now, in $K_v$ we take again the edges of $F+tq$, $F'+t-2(\alpha+\beta)-1$ and $F'+2t-2(\alpha+\beta)-1$, where $F$ and $F'$
are defined as before.
Furthermore, we take  also the edge $[m_{i,1},m_{i+1,1}]$ of length $1$, where $i$ is the  isolated vertex of $F$
(observe that we may always assume $i\geq 1$).
Finally, observe that, since $r'<t$ and $a+b<\frac{t-1}{2}$ it follows
$t-2(\alpha+\beta)-1>r'$ and so $F+tq$ and  $F'+t-2(\alpha+\beta)-1$ involve distinct rows.
\medskip

\noindent Case 2. $q$ odd.\\
Note that,  in this case, the  rows with an odd number of elements are the last $2t-r'=t-r$ rows of $M$.\\[2pt]
Case 2A. $a+b=t-\frac{r'+1}{2}$.\\
By Corollary \ref{Lemma12} there exists a \nf $F$ such that $\delta(F)=\{1^a, 2^b\}$
so it is sufficient to
 take the edges of $F+t-2(a+b)-1$ in  $K_v$.\\[2pt]
Case 2B. $a+b>t-\frac{r'+1}{2}.$\\
If $a+b-(t-\frac{r'+1}{2})$ is even there exist $\tilde a$ and $\tilde b$ such that $
\tilde a+\tilde b=t-\frac{r'+1}{2}$ and $a-\tilde a=2\alpha$ and $b-\tilde b=2\beta$, for suitable $\alpha$ and $\beta$.
Now let $F$ and $F'$ as in Case 1B. In $K_v$ we take the edges of $F+v-t$, $F'$ and $F'+t$.\\
If $a+b-(t-\frac{r'+1}{2})$ is odd there exist $\tilde a$ and $\tilde b$ such that $
\tilde a+\tilde b=t-\frac{r'+1}{2}$ and $a-\tilde a=2\alpha+1$ and $b-\tilde b=2\beta$, for suitable $\alpha$ and $\beta$.
In $K_v$ we take again the edges of $F+v-t$, $F'$ and $F'+t$. Finally we take the edge $[m_{i,1},m_{i+1,1}]$ where $i$ is the isolated vertex of $F+r$.
\end{proof}

 \begin{ex}\label{ex1}
 Let $L=\{1^2,2^3,12^9\}$, hence, following the notation of Proposition \ref{12t-cop}, we have
 $v=29$, $t=12$, $q=2$ and $r=r'=5$. It is easy to see that we are in
 Case 1B of the proof of Proposition \ref{12t-cop}.
 Let $M$ be the incomplete matrix defined in Proposition  \ref{12t-cop} (see Fig. \ref{tab1}).
 We take $\tilde a=\tilde b=1$ and
 $F=\{[0,1],[2,4]\}\cup\{3\}$. Now $\alpha=0$ and $\beta=1$ and we consider $F'=\{[0,2]\}\cup\{1\}$.
 Following the proof, we have to take the edges of $F+24$, $F'+9$ and
 $F'+21$, namely the following edges
 $\{[24,25],[26,28],[9,11],[21,23]\}$. Since the isolated vertex of $F$ is $3$ we have to take also
 the edge $[m_{3,1},m_{4,1}]=[2,3]$. Up to now we have constructed all the edges of length $1$ and $2$.
 It is easy to see that the elements not used to construct these edges
 appear in an even number on all but one row. Hence it is immediate to construct the edges of length $12$.\\
  Now let $L=\{1^4,2^2,12^{14}\}$, hence we have $v=41$, $t=12$, $q=3$, $r=5$ and $r'=17$.
 So we are in Case 2B of the proof of Proposition \ref{12t-cop}.
  Let $M'$ be the incomplete matrix defined in Proposition  \ref{12t-cop} (see Fig. \ref{tab1}).
 Since $a+b-(t-\frac{r'+1}{2})$ is odd, we take $\tilde a=1$, $\tilde b=2$ and we consider $F=\{[0,2],[1,3],[4,5]\}\cup \{6\}$.
 Now $\alpha=1$ and $\beta=0$, so we consider $F'=\{[0,1]\}\cup\{2\}$.
 Following the proof now we have to take the edges of $F+29$, $F'$ and $F'+12$, namely
 $[29,31],[30,32],[33,34],[0,1],[12,13]$.
 Finally, we take the edge $[m_{11,1},m_{12,1}]=[10,11]$
 since the isolated vertex of $F+r$ is $11$.
  Now it is trivial to construct the edges of length $12$.
 \end{ex}

  \begin{figure}[!t]
  \begin{center}
 $$M=\left(\begin{array}{ccc}
  0 &  12 &  {\bf24} \\
   1 & 13 & {\bf25}   \\
   {\bf2} & \emph{14} & {\bf26}  \\
 {\bf3} & 15 & 27 \\
  4 & 16 & {\bf28}  \\
  5 & 17 & \\
  6 & 18 & \\
  7 & 19 & \\
  8 & 20 & \\
 {\bf 9} & {\bf21} & \\
  10 & 22 & \\
  {\bf11} & {\bf23} &
   \end{array} \right)
 \quad \quad \quad\quad M'=\left(\begin{array}{cccc}
  {\bf0} &  {\bf12} &  24 & 36\\
   {\bf1} & {\bf13} & 25  & 37 \\
   2 & 14 & 26 & 38 \\
  3 & 15 & 27 & 39\\
  4 & 16 & 28 & 40 \\
  5 & 17 & {\bf29} & \\
  6 & 18 & {\bf30} & \\
  7 & 19 & {\bf31} & \\
  8 & 20 & {\bf32} & \\
  9 & 21 & {\bf33} & \\
  {\bf10} & \emph{22} & {\bf34} & \\
  {\bf11} & 23 & 35 &
   \end{array} \right)$$
\caption{Matrices of Example \ref{ex1}.}\label{tab1}
\end{center}
   \end{figure}

\begin{corollar}
If $t\leq 11$,
$\MPP(\{1^a,2^b,t^c\})$ holds for any $a,b,c\geq 0$.
\end{corollar}

\begin{proof}
For $t=3,4,5,6$ the thesis directly follows from Proposition \ref{big}.
If $t=7$ the thesis follows from Propositions \ref{big}, \ref{12t-notcop}
and \ref{12t-cop}.

Take now $t=8$, in view of Propositions \ref{big}
and \ref{12t-cop} we are left to consider the following classes of lists
$\{1,2,8^{8 q'+1}\}$ and $\{1,2,8^{8 q'+2}\}$.
If $L=\{1,2,8^{8q'+1}\}$, $v=16q'+7$. Let $q=2q'$ and $M$ be the incomplete matrix defined in Proposition
\ref{12t-cop}. Consider the edges $[m_{2,1},m_{1,q+1}]$ of length $8$,
$[m_{3,1},m_{4,1}]$ of length $1$ and $[m_{5,1},m_{7,1}]$ of length $2$.
Now all the rows of $M$ except the $6$-th one have an even number of adjacent entries so we can construct the remaining
edges of length $8$.\\
If $L=\{1,2,8^{8q'+2}\}$, $v=16q'+9$. Let $q=2q'+1$ and $M$ be the incomplete matrix defined in Proposition
\ref{12t-cop}. Consider the edges  $[m_{2,1},m_{3,q}]$ of length $8$,
$[m_{4,1},m_{5,1}]$ of length $1$ and $[m_{6,1},m_{8,1}]$ of length $2$.
Now all the rows of $M$ except the $7$-th one have an even number of adjacent entries so we can construct the remaining
edges of length $8$.

For $t=9$, in view of Propositions \ref{big}, \ref{12t-notcop}
and \ref{12t-cop} we are left to consider the following classes of lists
$\{1,2,9^{9q'+1}\}$ and $\{1,2,9^{9q'+3}\}$.
If $L=\{1,2,9^{9q'+1}\}$, $v=18q'+7$. Let $q=2q'$ and $M$ be the incomplete matrix defined in Proposition
\ref{12t-cop}. Consider the edges $[m_{3,1},m_{1,q+1}]$ of length $9$,
$[m_{2,1},m_{4,1}]$ of length $2$ and $[m_{5,1},m_{6,1}]$ of length $1$.
Now all the rows of $M$ except the $7$-th one have an even number of adjacent entries so we can construct the remaining
edges of length $9$.\\
If $L=\{1,2,9^{9q'+3}\}$, $v=18q'+11$. Let $q=2q'+1$ and $M$ be the incomplete matrix defined in Proposition
\ref{12t-cop}. Consider the edges $[m_{7,1},m_{9,q}]$ of length $9$,
$[m_{3,1},m_{4,1}]$ of length $1$ and $[m_{6,1},m_{8,1}]$ of length $2$.
Now all the rows of $M$ except the $5$-th one have an even number of adjacent entries so we can construct the remaining
edges of length $9$.

Finally, for $t=10,11$ it is easy to see that we are left to consider eight classes of lists for each value of $t$.
By direct computation we have checked that the conjecture holds also in these cases.
\end{proof}

To conclude we present an  example for lists not considered in previous propositions.
We consider the ``most difficult'' case namely when $a=b=1$.
It is important to underline that the strategy used for these
particular cases can be generalized to infinite classes of lists.

\begin{ex}\label{ex2}
Take $L=\{1,2,19^{23}\}$, hence $v=51$.
Let $M$ be the incomplete matrix defined in Proposition
\ref{12t-cop} (see Fig. \ref{tab2}) and take the following edges
$[38,6],[39,7],[40,$ $8], [41,9]$ of length $19$,
$[42,43]$ of length $1$ and $[48,50]$ of length $2$.
Now all the rows of $M$ except the $12$-th one  have an even number of adjacent entries so we can construct the remaining
edges of length $19$. It is easy to see that this construction can be easily generalized to any list
$\{1,2,19^{19k+4}\}$.\\
Take now $L=\{1,2,19^{28}\}$, hence $v=61$.
Let $M'$ be the incomplete matrix defined in Proposition
\ref{12t-cop} (see Fig. \ref{tab2}) and take the following edges
$[57,15],[58,16],$ $[59,17], [60,18],[0,42],[1,43],[2,44],[3,45],[8,50]$ of length $19$,
$[9,11]$ of length $2$ and $[13,14]$ of length $1$.
Now all the rows of $M'$ except the $11$-th one have an even number of adjacent entries so we can construct the remaining
edges of length $19$. It is easy to see that this construction can be easily generalized to any list
$\{1,2,19^{19k+9}\}$.
 \end{ex}
 
  \begin{figure}[!t]
  \begin{center}
 $$M=\left(\begin{array}{ccc}
  0 &  19 &  {\bf38} \\
   1 & 20 & {\bf39}   \\
   2 & 21 & {\bf40}  \\
 3 & 22 & {\bf41} \\
  4 & 23 & {\bf42}  \\
  5 & 24 & {\bf 43}\\
  {\bf6} & 25 & 44\\
  {\bf7} & 26 & 45\\
  {\bf8} & 27 & 46\\
 {\bf 9} & 28 & 47\\
  10 & 29 & {\bf 48}\\
  11 & 30 & \emph{49}\\
  12 & 31 & {\bf50}\\
  13 & 32 &\\
  14 & 33 & \\
  15 & 34 & \\
  16 & 35 & \\
  17 & 36 & \\
  18 & 37 &
   \end{array} \right)
 \quad \quad \quad\quad M'=\left(\begin{array}{cccc}
  {\bf0} &  19 &  38 & {\bf57}\\
   {\bf1} & 20 & 39  & {\bf58} \\
   {\bf2} & 21 & 40 & {\bf59} \\
  {\bf3} & 22 & 41 & {\bf60}\\
  4 & 23 & {\bf42} &  \\
  5 & 24 & {\bf43} & \\
  6 & 25 & {\bf44} & \\
  7 & 26 & {\bf45} & \\
  {\bf8} & 27 & 46 & \\
  {\bf9} & 28 & 47 & \\
  \emph{10} & 29 & 48 & \\
  {\bf11} & 30 & 49 & \\
  12 & 31 & {\bf50} & \\
  {\bf13} & 32 & 51 & \\
  {\bf14} & 33 & 52 & \\
  {\bf15} & 34 & 53 & \\
  {\bf16} & 35 & 54 & \\
  {\bf17} & 36 & 55 & \\
  {\bf18} & 37 & 56 &
   \end{array} \right)$$
   \caption{Matrices of Example \ref{ex2}.}\label{tab2}
\end{center}
   \end{figure}

We have to point out that the constructions illustrated in the previous example work for any values of
$t$ if $L=\{1,2,t^c\}$ with $c$ belonging to some suitable congruence classes modulo $t$.
Indeed for $t=19$ we are able to prove $\MPP(L)$ for any $c$, but the construction of the near $1$-factor depends
on the congruence class of $c$ modulo $19$.
In our opinion it seems not possible to present a general construction, but it is necessary to split
the proof of the remaining open case in several subcases.

\section*{Acknowledgments}

The authors thank  the anonymous referees for their useful suggestions and comments.


\begin{thebibliography}{99}

\bibitem{BA}
Baker C.A., Extended Skolem sequences. J. Combin. Des. 3, 363--379 (1995).

\bibitem{BBRT} Bonvicini S., Buratti M.,  Rinaldi G.,  Traetta T.,
Some progress on the existence of 1-rotational Steiner triple systems.
Des. Codes Cryptogr. 62 (2012), 63--78. 


%
\bibitem{BE}
Bryant D., El-Zanati S., Graph decompositions,
In: C.J. Colbourn and J.H. Dinitz editors, Handbook of combinatorial designs,
Second Edition, Chapman \& Hall/CRC, Boca Raton, FL, 2006, pp. 477--486.


\bibitem{BU}
Buratti M.,
1-rotational Steiner triple systems over arbitrary groups.
J. Combin. Des. 9, 215--226 (2001).




\bibitem{BM} Buratti M., Merola F.,
Dihedral Hamiltonian Cycle Systems
of the Cocktail Party Graph.
J. Combin. Des. 21, 1--23 (2013).


\bibitem{BMnew} Buratti M., Merola F.,
Hamiltonian cycle systems which are both cyclic and symmetric.
 J. Combin. Des. 22, 367--390 (2014).



\bibitem{CDF} Capparelli S., Del Fra A.,
Hamiltonian paths in the complete graph with edge-lengths $1,2,3$.
Electron. J. Combin. 17 (2010), $\sharp$R44.

\bibitem{D} Dinitz J.H., Starters, In:
Colbourn C.J. and Dinitz J.H. editors,
Handbook of Combinatorial Designs,
Second Edition,
Chapman \& Hall/CRC, Boca Raton, FL, 2006, pp. 622--628.

\bibitem{DJ} Dinitz J.H., Janiszewski S.R.,
On hamiltonian Paths with Prescribed edge lengths in the Complete Graph.
Bull. Inst. Combin. Appl. 57, 42--52  (2009).

\bibitem{FM}
Francetic N., Mendelsohn E.,
A survey of Skolem-type sequences and Rosa's use of them.
Math. Slovaca 59, 39--76 (2009).

%

\bibitem{GR} Godsil C., Royle G., Algebraic graph theory.
Graduate Texts in Mathematics. Vol 207. Springer, (2001).

\bibitem{HR} Horak P., Rosa A.,
On a problem of Marco Buratti.
Electron. J. Combin. 16 (2009), $\sharp$R20.

%

\bibitem{1235} Pasotti A.,  Pellegrini M.A., A new result on the problem of Buratti,
Horak and Rosa.  Discrete Math. 319, 1--14 (2014).

\bibitem{12t} Pasotti A., Pellegrini M.A., On the Buratti-Horak-Rosa conjecture about Hamiltonian paths in complete graphs.
Electron. J. Combi. 21 (2014), $\sharp$P2.30.

%
%
%
\bibitem{R} Rosa A., On a problem of Mariusz Meszka. Discrete Math. 338, 139--143 (2015).

\bibitem{S}
Shalaby N., Skolem and Langford Sequences,
In: C.J. Colbourn and J.H. Dinitz editors, Handbook of combinatorial designs,
Second Edition, Chapman \& Hall/CRC, Boca Raton, FL, 2006, pp. 612--616.

\bibitem{Wbook} West D., Introduction to graph theory. Prentice Hall, New Jersey (1996).


 \bibitem{W}  West D., http://www.math.uiuc.edu/$\thicksim$west/regs/buratti.html. Accessed 28 March 2014.
%
\bibitem{WB} Wu S.-L., Buratti M., A complete solution to the existence problem for 
$1$-rotational $k$-cycle systems of $K_v$. J. Combin. Des. 17 (2009), 283--293. 


\end{thebibliography}
\end{document}